\documentclass[12pt]{amsart}
\usepackage{mathrsfs}
\usepackage{txfonts}
\usepackage{amssymb}
\usepackage{amsmath}
\allowdisplaybreaks
\usepackage{color}
\makeatletter \@mparswitchfalse \makeatother
\newtheorem{theorem}{Theorem}[section]
\newtheorem{lemma}{Lemma}[section]

\newtheorem{proposition}{Proposition}[section]
\newtheorem{remark}{Remark}[section]
\newtheorem{example}{Example}[section]

\def\eqref#1{(\ref{eq#1})}

\numberwithin{equation}{section}
\begin{document}
\title{Order automorphisms of partial isometries in $M_n(\mathbb C)$}
\author[F. Jia, G. Ji]{Fengyang Jia, Guoxing Ji$^*$}
  \address{School  of Mathematics and Statistics,
  Shaanxi Normal University,
  Xian , 710119, People's  Republic of  China}
  \email{jfy123@snnu.edu.cn, gxji@snnu.edu.cn}

     \thanks{This research was
supported by the National Natural
   Science Foundation of China(No. 12271323).
\\ \indent $^*$Corresponding author}
 \maketitle
\begin{abstract}
  We  investigate and characterize order automorphisms on the set of partial isometries in the finite-dimensional matrix algebra $M_n(\mathbb{C})$. Different from the classical order automorphisms of subspace lattices, which can be implemented by standard invertible or unitary transformations, the order automorphisms considered herein admit no such conventional matrix representations. Instead, they are essentially governed by matrices such that $I-(A+A^*)$ is either positive or negative invertible. The results reveal that the structural features of order automorphisms for partial isometries are substantially more intricate than those of classical subspace automorphisms.
\end{abstract}

\baselineskip18pt

\section{Introduction}
A fundamental theorem of projective geometry \cite[p.~44]{bae} states that, for the complex Euclidean space $\mathbb{C}^n$ with $n\geq 3$, every lattice automorphism $\varphi$ on the set of all subspaces of $\mathbb{C}^n$ takes the form $\varphi(M)=SM$ for all subspaces $M\subseteq \mathbb{C}^n$. Here, $S$ is a $\tau$-quasilinear bijective transformation on $\mathbb{C}^n$, meaning that $\tau$ is a ring automorphism of $\mathbb{C}$ and $S$ is an additive bijection satisfying $S(ax)=\tau(a)S(x)$ for all $a\in\mathbb{C}$ and $x\in\mathbb{C}^n$.  P. A. Fillmore and W. E. Longstaff established an infinite-dimensional counterpart of this classical theorem for closed subspaces in \cite{fil}.

Identifying closed subspaces with orthogonal projections, this result essentially characterizes order automorphisms of the projection lattice. It is well known that every projection is a positive partial isometry, and the partial order defined on projections can be naturally extended to the set of all partial isometries \cite{hal}. In fact, this order is exactly the restriction of the operator star partial order to partial isometries(cf.\cite{do,dr}). A natural problem is to generalize such order isomorphism theorems from projections to the whole set of partial isometries. Combining the partial order with orthogonality, L. Moln$\acute{\text{a}}$r in \cite{mol} characterized order-orthogonal automorphisms under a mild continuity assumption, while the authors in \cite{sh} removed the continuity condition and obtained the same characterization of automorphisms without additional restrictions.

In fact, the study of general order automorphisms is of great research significance. This paper focuses on characterizing order automorphisms of partial isometries in the matrix algebra $M_n(\mathbb{C})$. We show that, in contrast to the classical subspace order automorphism characterization, such automorphisms cannot be realized by unitary or invertible matrices in the usual sense. Instead, they are essentially determined by matrices for which $I-(A+A^*)$ is either positive invertible or negative invertible. Accordingly, the structure of these order automorphisms is far more complicated than the classical cases. We now recall some basic definitions and preliminary results.

Throughout this paper, let $\mathbb{C}^n$ denote the $n$-dimensional complex Euclidean space with $n\geq 3$ and let $M_n(\mathbb{C})$ stand for the algebra of all $n\times n$ complex matrices. For any $x,y\in\mathbb C^n$, $\langle x,y\rangle$ stands for the inner product of $\mathbb C^n$ and $y\otimes x $ is the rank one operator defined by $y\otimes x(z)=\langle z,y\rangle x,$ for any $z\in\mathbb C^n$. For a subset $S\subseteq \mathbb C^n$, $[S]$ denotes the linear space spanned by $S$ and $S^{\perp}$ is the orthogonal complement of $S$.  For any $A\in M_n(\mathbb C)$, we denote by $A^*$  and $\sigma(A)$ the adjoint and eigenvalue of $A$ respectively.  A projection $P\in M_n(\mathbb C)$ is a matrix satisfying $P=P^*=P^2$ and $\mathcal P(M_n(\mathbb C))$  denotes the set of all projections in $M_n(\mathbb C)$. A matrix $U\in M_n(\mathbb C)$ is called a partial isometry
if $U^*U\in \mathcal P(M_n(\mathbb C))$, which is called the initial space of $U$. In this case,  $UU^*$ is also a projection and is called the final projection.  The set of all partial isometries in $M_n(\mathbb C)$ is denoted by $\operatorname{PI}(M_n(\mathbb C)$.

In particular, all aforementioned notions coincide with their counterparts in general Hilbert spaces, and throughout this paper we restrict our attention to the finite-dimensional case.
 \section{Preliminaries}
 Let $A,B\in M_n(\mathbb{C})$. Recall that we define the star partial order $A\overset{*}{\leq} B$ by the conditions
\(
A^*A = A^*B \quad \text{and} \quad AA^* = AB^*
\)
(see \cite{do,dr}). Suppose now that $A,B\in  \operatorname{PI} (M_n(\mathbb{C}))$.
Then $A\overset{*}{\leq} B$ holds if and only if
\[
A^*A \leq B^*B,\quad AA^* \leq BB^* \quad \text{and} \quad BA^*A = A,
\]
meaning that $B$ coincides with $A$ on the initial space of $A$ (see \cite{hal,mol}).
For $A,B\in  \operatorname{PI} (M_n(\mathbb{C}))$, we write $A\leq B$ as a shorthand for this star order.
It is straightforward to verify that $ \mathcal P (M_n(\mathbb{C})) \subseteq  \operatorname{PI} (M_n(\mathbb{C}))$,
and under this inclusion, the standard partial order on projections coincides exactly with the restricted partial order inherited from $ \operatorname{PI} (M_n(\mathbb{C}))$. For any rank one partial isometry $y\otimes x$, we assume that $\|x\|=\|y\|=1$. $y\otimes x \leq U$ for some partial isometry $U$ if and only if $Ux=y$. The lemma below is straightforward. We provide its proof for the sake of completeness.
\begin{lemma} \label{lemma21}
Suppose $y_k\otimes x_k$ ($k=1,2$) are two rank-one partial isometries.
There exists a partial isometry $R$ satisfying $y_k\otimes x_k\le R$ for $k=1,2$ if and only if
\(
\langle y_1,y_2\rangle = \langle x_1,x_2\rangle.
\)
\end{lemma}
\begin{proof}  We only prove the sufficiency. We may assume that $x_1$ and $x_2$  are linearly independent. Put $x_1 = \langle x_1,x_2\rangle x_2 + \xi x_3$ and $y_1 = \langle y_1,y_2\rangle y_2 + \eta y_3$
for unit vectors $x_3$ and $y_3$ satisfying $x_3\perp x_2$ and $y_3\perp y_2$. Then $|\xi|^2=1-|\langle x_1,x_2\rangle|^2=1-|\langle y_1,y_2\rangle|^2=|\eta|^2$. We now define a partial isometry $R$ by
 $Rx_2=y_2 $,  $Rx_3=\eta \xi^{-1} y_3$ and $Rx=0$ for all $x\in [x_2,x_3]^{\perp}$. Then
  \[Rx_1=R(\langle x_1,x_2\rangle x_2+\xi x_3)=\langle y_1,y_2\rangle y_2+\xi ( \eta \xi^{-1} y_3)=y_1.\]
  Hence $y_k\otimes x_k=Rx_k\otimes x_k\leq R$ for $k=1,2$.
\end{proof}
In general, let $U, V, W\in \operatorname{PI}(M_n(\mathbb{C}))$.
If $U \leq W$ and $V \leq W$ (resp. $W \leq U$ and $W \leq V$), we call $W$ an upper bound (resp. a lower bound) for $U$ and $V$. The supremum and infimum may be defined analogously.
When $U^*V = UV^* = 0$, we say that the partial isometry $U$ is orthogonal to the partial isometry $V$ and denoted by $U\perp V$.
Just as in the setting of the star partial order, the following proposition has a straightforward proof, so we state it without proof.
\begin{proposition}\label{prop21} Let $U, V\in \operatorname{PI}(M_n(\mathbb{C}))$.

$(i)$  If $U$ and $V$ admit an upper bound (resp. a lower bound), then the supremum $\sup\{U,V\}$ (resp. the infimum $\inf\{U,V\}$) exists.

$(ii)$ $U\leq U+V$ if and only if $U\perp V$.

$(iii)$ $U\leq V$ if and only if $U=VU^*U$. In particular, $\{U:U\leq V\}=\{VE: E\leq V^*V\}$.
\end{proposition}
 As stated by the fundamental theorem of projective geometry \cite[p.~44]{bae}, every order automorphism of the projection lattice
  $\mathcal P(M_n(\mathbb{C}))$ is induced by a $\tau$-quasilinear invertible transformation $S$ on $\mathbb{C}^n$.
Order automorphisms of $\operatorname{PI}(M_n(\mathbb{C}))$ are closely linked to $\tau$-quasilinear invertible transformations on $\mathbb{C}^n$.
We begin with the following proposition.
  Let    $\tau$ be a ring automorphism of  $\mathbb C$. A transformation $A$ on $\mathbb C^n$ is said to be $\tau$ quasi-linear  if $A(ax+y)=\tau(a)Ax+  Ay$ for all $a\in\mathbb C$ and $x,y\in\mathbb C^n$. If $\tau(a)=\bar{a}$ for all $a\in\mathbb C$, then $A$ is said to be conjugate linear. In particular, when $A^*A=AA^*=I$, we say that $A$ is a  conjugate unitary transformation.
\begin{proposition}  \label{prop22}
Let $\tau$ and $\delta$ be ring automorphisms of $\mathbb{C}$, and let $A$ and $B$ be invertible $\tau$-quasilinear and $\delta$-quasilinear transformations on $\mathbb{C}^n$, respectively.

\noindent $(i)$ If $[Ax] = [Bx]$ for all $x\in\mathbb{C}^n$, then $\delta = \tau$ and $B = bA$ for some nonzero constant $b\in\mathbb{C}$.

\noindent $(ii)$ If the equivalence
    \[
    \langle Ax,y\rangle = 0 \iff x\perp y
    \]
    holds for all $x,y\in\mathbb{C}^n$, then $A = aI$ for some constant $a\in\mathbb{C}$.

\noindent $(iii)$ Suppose that both $A$ and $B$ are $\tau$-quasilinear transformations satisfying
    \[
    \langle Ax,By\rangle = 0 \iff x\perp y
    \]
    for all $x,y\in\mathbb{C}^n$. Then both $A$ and $B$ are either linear or conjugate linear.
\end{proposition}
\begin{proof} $(i)$ It is elementary that $Bx=b_xAx$ for some $b_x\in\mathbb C$. Take any two linearly independent vectors $x,y\in\mathbb C^n$,  we have
$$B(x+y)=b_{x+y}A(x+y)=b_{x+y}Ax+b_{x+y}Ay=Bx+By=b_xAx+b_yAy.$$ It easily follows that $b_x=b_y$. Thus $Bx=bAx$ for all $x\in\mathbb C^n$. In particular, for any $a\in\mathbb C$,
$\delta (a)Bx=Bax=bAax=b\tau(a)Ax=\tau(a)bAx=\tau(a)Bx$.  Then $\delta=\tau$.

$(ii)$ Note that $[Ax]=[Ix]$ for all $x\in\mathbb C^n$. By $(i)$, $A=aI$ for some constant $a\in\mathbb C$.

$(iii)$ Fixed  any  $y\in\mathbb C^n$ and put $f_y(x)=\tau^{-1}(\langle Ax,By\rangle)$ for any $x\in\mathbb C^n$. Then $f_y$ is a linear functional on $\mathbb C^n$. Thus there exists an element $C_By\in \mathbb C^n$ such that $f_y(x)=\langle x, C_By\rangle$ for all $x\in\mathbb C^n$. It is elementary that
$C_B(y_1+y_2)=C_By_1+C_By_2$ for all $y_1,y_2\in\mathbb C^n$. For any $a\in\mathbb C$,
\begin{align*}\label{formula21}
 &\ \ \ \  \langle x,C_b ay\rangle = f_{ay}(x)
 =\tau^{-1}(\langle Ax,Bay\rangle)=\tau^{-1}(\langle Ax, \tau(a)By\rangle)\\
\nonumber&=\tau^{-1}\left(\overline{\tau(a)}\langle Ax, By\rangle\right)=\tau^{-1}\left(\overline{\tau(a)}\right)\tau^{-1}(\langle Ax, By\rangle)
=\tau^{-1}\left(\overline{\tau(a)}\right)f_y(x) \\
\nonumber& =\tau^{-1}\left(\overline{\tau(a)}\right) \langle x, C_By\rangle=\langle x,\overline{\tau^{-1}\left(\overline{\tau(a)}\right)}C_By\rangle.
\end{align*}
Let $\delta(a)=\overline{\tau^{-1}\left(\overline{\tau(a)}\right)}$ for all $a\in\mathbb C$. Then $\delta $ is a ring automorphism of $\mathbb C$  and $C_Bay=\delta(a)C_By$. That is, $C_B$ is a $\delta$-quasilinear transformation  on $\mathbb C$. It is elementary that
 $C_B$ is invertible. Note that  for any $x,y\in\mathbb C^n$,
\[ \langle x, C_By\rangle=0 \iff f_y(x)=\langle Ax, By\rangle=0 \iff \langle x,y\rangle=0. \]
 By $(ii)$, we have $C_B=cI $ for some constant. In particular, $C_B$ is linear and  \(\delta(a)=\overline{\tau^{-1}\left(\overline{\tau(a)}\right)}=a\) for all $a\in \mathbb C$. Hence
 $$\tau^{-1}\left(\overline{\tau(a)}\right)=\bar{a}, \ \forall a\in\mathbb C. $$
This implies that
$$ \overline{\tau(a)}=\tau(\bar{a}), \ \forall a\in\mathbb C.$$ Thus $\tau(\mathbb R)=\mathbb R$ and $\tau(i^2)=(\tau(i))^2=\tau(-1)=-1$.
Hence $\tau(i)=i$, $\tau(a)=a$ for all $a\in\mathbb C$ or $\tau(i)=-i$,  $\tau(a)=\bar{a}$ for all $a\in\mathbb C$. Consequently, $A$ and $B$ are linear or conjugate linear. \end{proof}
\section{Order automorphisms of $\operatorname{PI}(M_n(\mathbb C))$}
   Let $\varphi$ be an order automorphism of $\operatorname{PI}(M_n(\mathbb{C}))$, that is, a bijection satisfying
    \[
    \varphi(U) \le \varphi(V) \iff U \le V
    \]
    for all $U,V\in \operatorname{PI}(M_n(\mathbb{C}))$. Let $C$ and $D$ be two  unitary   or two  conjugate unitary transformations on $\mathbb{C}^n$. We define the map $\Delta_{CD}$ by
    \[
    \Delta_{CD}(T) = CTD,\quad \forall\, T\in \operatorname{PI}(M_n(\mathbb{C})).
    \]
    We now establish several fundamental properties of such order automorphisms $\varphi$.
\begin{proposition}\label{prop31}
    Let $\varphi$ be an order automorphism of $\operatorname{PI}(M_n(\mathbb{C}))$. Then the following statements hold:

   $(i)$ $\varphi(U)$ is unitary if and only if $U$ is unitary.

         $(ii)$ $\operatorname{r}(\varphi(U)) = \operatorname{r}(U)$ for all $U\in \operatorname{PI}(M_n(\mathbb{C}))$, where $\operatorname{r}(U)$ denotes the rank of $U$.

         $(iii)$
         For any two unitary matrices or conjugate unitary transformations $C,D$ on $\mathbb{C}^n$, the map $\Delta_{CD}$ is an order automorphism of $\operatorname{PI}(M_n(\mathbb{C}))$.
         \end{proposition}
  Now we put $U=\varphi(I)$ and  define the map $\Phi =\Delta_{U^*I}\circ \varphi$. Then $\Phi$ is also an order automorphism satisfying $\Phi(I)=I$ by Proposition \ref{prop31}. Since $E\leq I$ if and only if $E$ is a projection,  $\Phi(E)$ is a projection if and only if $E$ is a projection. Thus  the restriction $\Phi|_{\mathcal P(M_n(\mathbb C))}$ is an order automorphism  of \(\mathcal P(M_n(\mathbb C))\).
  By the fundamental theorem of elementary projective geometry  \cite[p.~44]{bae}, there exist a ring isomorphism $\tau:\mathbb{C}\to\mathbb{C}$ and a $\tau$-quasilinear invertible transformation $S$ on $\mathbb{C}^n$ such that $\Phi(E)=P_{SE}$ for all $E\in \mathcal P(M_n(\mathbb C))$, where $P_M$ denotes the projection on the subspace $M$. For each $a\in\mathbb{T}$, the unit circle $\{a\in\mathbb C:|a|=1\}$, we have $\Phi(aI)=U_a$ is  also unitary. We define a new map $\Phi_a(T) = U_a^*\Phi(aT)$ for all $T\in  PI (M_n(\mathbb C))$ again. It is straightforward to verify that $\Phi_a(I)=I$.
By a similar argument, for each $a\in\mathbb{T}$, there exist a ring automorphism $\tau_a $ of $\mathbb{C}$ and a $\tau_a$-quasilinear invertible transformation  $S_a$ on $ \mathbb{C}^n$ such that $\Phi_a(E)=P_{S_aE}$ for all $E\in \mathcal P(M_n(\mathbb C))$. In particular, for every projection $E\in \mathcal P(M_n(\mathbb C))$ and every scalar $a\in\mathbb{T}$, we obtain
    \begin{equation}\label{formula31}
    \Phi(aE)=U_a\Phi_a(E)=U_aP_{S_aE}.
    \end{equation}
\begin{lemma}\label{lemma31}
Either every $S_a$ ($a\in\mathbb T$) is linear, or every $S_a$ is conjugate linear.
\end{lemma}
\begin{proof}
For any unit vectors $x,y \in\mathbb C^n$ and $a,b\in\mathbb T$ with $a\ne b$, it is elementary by Lemma \ref{lemma21}
  that
\[
\sup\{ax\otimes x,by\otimes y\}\text{ exists } \iff x\perp y.
\]
Note that
\(
\Phi(a x\otimes x)=U_a\frac{S_ax}{\|S_ax\|}\otimes \frac{S_ax}{\|S_ax\|}
\) by formula (\ref{formula31}).
Then
\[
\sup\left\{U_a\frac{S_ax}{\|S_ax\|}\otimes \frac{S_ax}{\|S_ax\|},U_b\frac{S_by}{\|S_by\|}\otimes \frac{S_by}{\|S_by\|} \right\}\text{ exists} \iff x\perp y.
\]
It follows from Lemma \ref{lemma21} that
\[
\langle U_aS_ax,U_bS_by\rangle =\langle S_ax,S_by\rangle \iff x\perp y.
\]
That is,
\begin{equation}\label{formula32}
\langle (U_b^*U_a-I)S_ax,S_by\rangle =0 \iff x\perp y.
\end{equation}
Moreover, it is elementary that $\inf\{aI,bI\}$ does not exist. It follows that $\inf\{U_a,U_b\}$ does not exist and hence  $1\not\in \sigma(U_b^*U_a)$ for any $a,b\in\mathbb T$ with $a\ne b$.  Fixed a constant  $b$ and take any two  constants  $a,c \in\mathbb T\setminus\{1\}$. By formula (\ref{formula32}),
we have
\begin{equation}\label{formula33}
[(U_b^*U_a-I)S_ax]=[(U_b^*U_c-I)S_cx]=[S_b[x]^{\perp}]^{\perp}.
\end{equation}
Since $(U_b^*U_a-I)S_a$ and $(U_b^*U_c-I)S_c$ are bijective, $\tau_a=\tau_c$ and $(U_b^*U_a-I)S_a=\lambda (U_b^*U_c-I)S_c$ for some constant $\lambda $ by Proposition \ref{prop22}. Note that $b$ is also arbitrary. we have $\tau_a=\tau$ for all $a\in \mathbb T$. Thus  every $S_a$ is $\tau$ quasi-linear.

Moreover, by formula (\ref{formula32}) and Proposition \ref{prop22}, $S_a$ and $S_b$ are either both linear or both conjugate linear.   Let $b=1$. If $S$ is linear, then $S_a$ is linear for all $a\in\mathbb{T}$. If $S$ is conjugate linear, then $S_a$ is conjugate linear for all $a\in\mathbb{T}$.
 \end{proof}
  Let $U\in M_n(\mathbb C)$ be an arbitrary unitary matrix. Define the map $\Phi_U\colon \operatorname{PI}(M_n(\mathbb C))\to \operatorname{PI}(M_n(\mathbb C))$ by
    \[
    \Phi_U(T) = \Phi(U)^*\Phi(UT)
    \]
    for all $T\in \operatorname{PI}(M_n(\mathbb C))$. Then $\Phi_U$ is also an order automorphism on $\operatorname{PI}(M_n(\mathbb C))$ and satisfies $\Phi_U(I)=I$. By the same argument as before, there exists a $\tau$-quasilinear invertible transformation $S_U$ on $\mathbb{C}^n$ such that $\Phi(UE)=\Phi(U)P_{S_U E}$ for every projection $E\in \mathcal P(M_n(\mathbb C))$. In the following, we prove that $S_U$ is also  linear.
 \begin{lemma}  Assume that $S_a$ is linear for all $a\in\mathbb{T}$. Then $S_U$ is a linear invertible  transformation on $\mathbb{C}^n$ for every unitary matrix $U\in \operatorname{PI}(M_n(\mathbb{C}))$.
 \end{lemma}
 \begin{proof} For any two unitary matrices $U,V\in \operatorname{PI}(M_n(\mathbb C))$ and unit vectors $x,y\in\mathbb C^n$, the following equivalence holds by Lemma \ref{lemma21}.
 \[\langle  \Phi(U)S_Ux, \Phi(V)S_Vy\rangle =\langle S_Ux,S_Vy\rangle \iff \langle Ux,Vy\rangle=\langle x,y\rangle.\]
 Then \begin{equation}\label{formula34}
 \langle \bigl((\Phi(V))^*\Phi(U)-I\bigr)S_Ux, S_Vy \rangle =0 \iff \langle (V^*U-I)x,y\rangle=0.\end{equation}
 In particular,
 \begin{equation}\label{formula35}
 \langle  ((U_a)^*\Phi(U)-I)S_Ux, S_ay \rangle =0 \iff \langle (\bar{a}U-I)x,y\rangle=0\end{equation}
  for any $a\in\mathbb T$. Note that if $a\not\in \sigma (U)$, then $1\not\in\sigma\left((U_a)^*\Phi(U)\right)$.  Otherwise, there exists a unit vector $y\in \mathbb C^n$ such that
 $\Phi(U)y=U_ay$. This means that  $\Phi(U)y\otimes y=U_ay\otimes y\leq \inf\{\Phi(U),U_a\}$.  Thus
 $\Phi^{-1}(U_ay\otimes y)\leq \inf\{U,aI\}$. That is, there exists a unit vector $x\in\mathbb C^n$  such that
 $ax\otimes x\leq U$. This implies that  $Ux=ax$ and  $a\in\sigma (U)$, a contradiction. Hence we may take a constant $a\in\mathbb T$ such that $\bar{a}U-I$ is invertible and so is $(U_a)^*\Phi(U)-I$. Since $S_a$ is linear, by formula (\ref{formula35}), we have
 \begin{equation*}\label{formula36}
 \langle  (S_a^*((U_a)^*\Phi(U)-I)S_Ux, y \rangle =0 \iff \langle (\bar{a}U-I)x,y\rangle=0\end{equation*} for any $x,y\in\mathbb C^n$.
 Thus $[(S_a^*((U_a)^*\Phi(U)-I)S_Ux]=[(\bar{a}U-I)x]$ for any $x\in\mathbb C^n$.  By Proposition \ref{prop22},
  \begin{equation}\label{formula36}
  S_a^*((U_a)^*\Phi(U)-I)S_U=\lambda (\bar{a}U-I)\end{equation}
   for some constant $\lambda$. In particular,  $S_U$ is linear.
 \end{proof}
 \begin{remark}\label{remark31} \upshape   If $S_a$ is conjugate linear for every $a\in\mathbb T$ by Lemma \ref{lemma31},   then  $S_U$ is also conjugate linear for every unitary matrix $U$.  In both cases,  let $S=U|S|$ be the polar decomposition of $S$, where $U$ is linear or conjugate linear and $|S|$ is linear.  Put $\Psi =\Delta_{U^*U}\circ \Phi$. Then we also have  $\Psi$ is an order automorphism of $\operatorname{PI}(M_n(\mathbb C))$ satisfying  $\psi(I)=I$ and $\Psi(E)=U^*P_{SE}U=P_{U^*SE}=P_{|S|E}$ for any projection $E$. Thus without loss of generality, we may assume that $S$ is a   positive invertible matrix in $M_n(\mathbb C)$.
 \end{remark}
 \begin{remark}\label{remark32}\upshape
 By formula (\ref{formula36}), there exists a scalar $\lambda_U$ such that
\[
S(\Phi(U)-I)S_U = \lambda_U(U-I)
\]
holds for every unitary matrix $U$. Since the equality $[S_U E] = [\lambda S_U E]$ is valid for any nonzero scalar $\lambda$, we may replace $S_U$ with $\lambda^{-1}_US_U$ if needed. After this normalization, we may assume
\begin{equation}\label{formula37}
S(\Phi(U)-I)S_U = U-I
\end{equation}
for all unitary matrices $U$. In particular, for each $a\in\mathbb{T}$, we obtain
\begin{equation}\label{formula38}
S(U_a-I)S_a = (a-1)I.
\end{equation}
 \end{remark}
\begin{lemma} \label{lemma33}
 Let $S$ be positive invertible. There exists a real constant $d$ such that for every non-identity unitary matrix of the form
\[U=\sum_{k=1}^n a_k e_k\otimes e_k,\]
where $a_k\in\sigma (U)$ and $\{e_k\}_{1\le k\le n}$ is an orthogonal basis of $\mathbb C^n$, the following two conditions hold:

(1) For all vectors $x=\sum_{k=1}^n x_k e_k\in\mathbb C^n$,
\[\begin{aligned} S_U x&=\sum_{a_k\ne 1} x_k S_{a_k} e_k + d\,S\sum_{a_k=1} x_k e_k, \\ \Phi(U)S_U x&=\sum_{a_k\ne 1} x_k U_{a_k} S_{a_k} e_k + d\,S\sum_{a_k=1} x_k e_k. \end{aligned}\]

(2) For every non-identity unitary matrix $V$,
\[S_V^*\big(\Phi(V)^*\Phi(U)-I\big)S_U = d\big(V^*U-I\big).\]
\end{lemma}
\begin{proof}
Note that $a_ke_k\otimes e_k\leq \inf\{U,a_kI\}$ for every $1\leq k\leq n$. Then
\[\Phi(a_ke_k\otimes e_k)=U_{a_k}\frac{1}{\|S_{a_k}x\|^2}S_{a_k}e_k\otimes S_{a_k}e_k=\Phi(U)\frac{1}{\|S_{U}x\|^2}S_{U}e_k\otimes S_{U}e_k.\]
It follows that
$S_{U}e_k=\xi_k S_{a_k}e_k$ and $\Phi(U)S_Ue_k=\xi_k U_{a_k}S_{a_k}e_k$ for some $\xi_k\in\mathbb C$ and  every $1\leq k\leq n$ . Thus
\[S_Ux=\sum\limits_{k=1}^nx_kS_Ue_k=\sum\limits_{k=1}^nx_k\xi_kS_{a_k}e_k\] and
\[\Phi(U)S_Ux=\sum\limits_{k=1}^nx_k\Phi(U)S_Ue_k=\sum\limits_{k=1}^nx_k\xi_kU_{a_k}S_{a_k}e_k.\]
\begin{align*}
&\ \ \ \ S(\Phi(U)-I)S_Ux=\sum\limits_{k=1}^n x_k\xi_kS(U_{a_k}-1)S_{a_k}e_k=\sum\limits_{k=1}^nx_k\xi_k(a_k-1)e_k\\
&=(U-I)x=\sum\limits_{k=1}^nx_k(U-I)e_k=\sum\limits_{k=1}^nx_k(a_k-1)e_k.
\end{align*}
Thus $\xi_k=1$ for every $1\leq k\leq n$ with $a_k\ne 1$. On the other hand, $S_Ux=\lambda Sx$ for all $x\in [e_k: a_k=1]$. Thus
$\xi_k=\xi_j$ for all $a_k=a_j=1$.  Put $\xi_U=\xi_k$ for all $k$ with $a_k=1$.

By formula (\ref{formula36}), we have $[S_V^*((\Phi(V))^*\Phi(U)-I)S_Ux]=[(V^*U-I)x]$ for all $x\in\mathbb C^n$. Then by \cite[Theorem 2.3]{br},
$R(S_V^*((\Phi(V))^*\Phi(U)-I)S_U)= R(V^*U-I)=[e ]$  for a nonzero vector $e$ or  $S_V^*((\Phi(V))^*\Phi(U)-I)S_U=\beta(U,V)(V^*U-I)$ for some nonzero constant $\beta(U,V)$. If $V^*U-I=\alpha e\otimes e$ is of  rank 1, then  for any $x\in [e]^{\perp}$, $(S_V^*((\Phi(V))^*\Phi(U)-I)S_U)x=0$.
It follows that $S_V^*((\Phi(V))^*\Phi(U)-I)S_U=\beta(U,V)(V^*U-I)$ again. In both cases,
\begin{equation}\label{formula39}
S_V^*((\Phi(V))^*\Phi(U)-I)S_U=\beta(U,V)(V^*U-I)\end{equation}
 for any non-identity unitary matrices $U$ and $V$. In particular,
 \begin{equation}\label{formula310}
 S_b^*(U_b^*U_a-I)S_a=\beta(a,b)(\bar{b}a-1)I\end{equation} for all $a,b \in\mathbb T$ with $a\ne b$. It is elementary that
  $\beta(a,1)=\beta(1,a)=1$ for every  $a\in\mathbb T$ with $a\ne 1$ by formula (\ref{formula38}).
Let $U=\sum\limits_{k=1}^na_ke_k\otimes e_k$ and $V=\sum\limits_{j=1}^nb_jf_j\otimes f_j$.
Take any $x,y\in\mathbb C^n$ such that $x=\sum\limits_{k=1}^nx_ke_k$ and $y=\sum\limits_{j=1}^ny_jf_j$. We have
\begin{align}\label{formula311}
&\ \ \ \ \langle S_V^*(\Phi(V))^*\phi(U)S_Ux, y\rangle=\langle \phi(U)S_Ux,\Phi(V)S_Vy\rangle\\
\nonumber &=\langle \sum\limits_{k=1}^nx_k\Phi(U)S_Ue_k, \sum\limits_{j=1}^ny_j\Phi(V)S_Vf_j\rangle\\
\nonumber &=\sum\limits_{k=1}^n \sum\limits_{j=1}^n\xi_kx_k\overline{\eta_jy_j}\langle U_{a_k}S_{a_k}e_k, U_{b_j}S_{b_j}f_j  \rangle \\
\nonumber &=\sum\limits_{k=1}^n \sum\limits_{j=1}^n\xi_kx_k\overline{\eta_jy_j}\langle S_{b_j}^*U_{b_j}^*U_{a_k}S_{a_k}e_k,  f_j  \rangle \\
\nonumber &=\sum\limits_{k=1}^n \sum\limits_{j=1}^n\xi_kx_k\overline{\eta_jy_j}\langle \left(S_{b_j}^* S_{a_k}+(\beta(a_k,b_j)\left(\overline{b_j}a_k-1\right)I\right)e_k,  f_j  \rangle \\
\nonumber &=\sum\limits_{k=1}^n \sum\limits_{j=1}^n\xi_kx_k\overline{\eta_jy_j}\langle S_{b_j}^* S_{a_k} e_k,  f_j  \rangle
 +\sum\limits_{k=1}^n \sum\limits_{j=1}^n\xi_kx_k\overline{\eta_jy_j}\langle  (\beta(a_k,a_j)\left(\overline{b_j}a_k-1\right)I e_k,  f_j  \rangle \\
 \nonumber &=\langle \sum\limits_{k=1}^n \xi_kx_k S_{a_k} e_k, \sum\limits_{j=1}^n  \eta_jy_j  S_{b_j}  f_j  \rangle
 +\sum\limits_{k=1}^n \sum\limits_{j=1}^n\xi_kx_k\overline{\eta_jy_j}\langle  (\beta(a_k,b_j)\left(\overline{b_j}a_k-1\right)I e_k,  f_j  \rangle \\
\nonumber  &=\langle S_Ux,S_Vy\rangle +\sum\limits_{k=1}^n \sum\limits_{j=1}^n\xi_kx_k\overline{\eta_jy_j}  (\beta(a_k,b_j)\left(\overline{b_j}a_k-1\right) \langle e_k,  f_j  \rangle,
\end{align}
where $\xi_k=\eta_j=1$ if $a_k \ne1$ and  $b_j\ne 1$,  $\xi_k=\xi_U$ and  $\eta_j=\xi_V$ if $a_k=b_j=1$.
On the other hand,
\begin{align}\label{formula312}
&\ \ \ \ \langle S_V^*(\Phi(V))^*\Phi(U)S_Ux, y\rangle=\langle \left(S_V^*S_U+\beta(U,V)(V^*U-I)\right)x,y\rangle\\
\nonumber &=\langle S_Ux,S_Vy\rangle +\langle  \beta(U,V)(V^*U-I)x,y\rangle
\end{align}and
\begin{align}\label{formula313}
& \ \ \ \ \langle  \beta(U,V)(V^*U-I)x,y\rangle =\beta(U,V)\left(\langle Ux,Vy\rangle -\langle x,y\rangle
\right) \\
\nonumber & =\beta(U,V)\left( \langle \sum\limits_{k=1}^n x_k a_k e_k ,\sum\limits_{k=1}^ny_j b_j f_j\rangle -\langle \sum\limits_{k=1}^nx_ke_k,\sum\limits_{j=1}^ny_jf_j\rangle\right)\\
\nonumber &=\beta(U,V)\left(\sum\limits_{k=1}^n\sum\limits_{j=1}^n( \overline{b_j}a_k-1)x_k\overline{y_j}\langle e_k,f_j\rangle
 \right).
\end{align}
Thus
\begin{align}\label{formula314}
&\ \ \ \  \sum\limits_{k=1}^n \sum\limits_{j=1}^n\xi_kx_k\overline{\eta_ky_j}  \beta(a_k,b_j)\left(\overline{b_j}a_k-1\right)\langle e_k,  f_j  \rangle\\
\nonumber &= \sum\limits_{k=1}^n\sum\limits_{j=1}^n x_k\overline{y_j}\beta(U,V)( \overline{b_j}a_k-1)\langle e_k,f_j\rangle
 .
\end{align}
It follows that
\begin{equation}\label{formula315}
\beta(U,V)=\beta(a_k,b_j)\xi_k \overline{\eta_j}
\end{equation} for all pairs $(k,j)$ such that \( \left(\overline{b_j}a_k-1\right)\langle e_k,  f_j  \rangle\ne0\).

 For any $a_1,b_1,a_2,b_2\in\mathbb T\setminus\{1\}$ with $a_k\ne b_k$ for $k=1,2$,  we define two unitary operators $U_0e_k=a_ke_k$,  $V_0e_k=b_ke_k$
  for $k=1,2$ and  $U_0e_j=e_j$, $V_0e_j=ce_j $ for $3\leq j\leq n$, where $c\in\mathbb T\setminus\{1\}$. Then $\beta(U_0,V_0)=\beta(a_1,b_1)=\beta(a_2,b_2)=\beta(c,1)\xi_{U_0}=\xi_{U_0}$ by formula (\ref{formula315}).

  This means that $\beta(a,b)$ is a constant  for all $a,b\in\mathbb T\setminus\{1\}$ with $a\ne b$.  We also note that $\beta(a,b)=\overline{\beta(b,a)}$. Then $\beta(a,b)$ is a real constant.   Put $d=\beta(a,b)$.

 Now for any unitary matrices $U$ and $V$,
 since $U\ne V$, there exists at least a pair $(k,j)$ such that $\beta(U,V)=\beta(a_k,b_j)\xi_k\overline{\eta_j}\ne 0$ by formula (\ref{formula315}). If neither $a_k$ nor $b_j$ is 1 for such a pair $a_k$ and $ b_j$,  then $\beta(U,V)=\beta(a_k,b_j)=d$. Otherwise,  let   $a_k=1$ and $b_j\ne 1$.
 In this case, $\beta(U,V)=\beta(1,b_j) \xi_U=\xi_U$.    Note that there is a $t$ such that $a_t\ne1$.
 We  define a unitary matrix $W$
 such that $We_k=ce_k $, $We_t=ce_t$ and $We_m=e_m$ for and $ m\not\in \{k, t\}$,   where $c \in\mathbb T$ $a\not =a_t$. Then
  $\beta(U,W)=\beta(1,c)\xi_U=\xi_U=\beta(a_t,c)=d$. Hence $\beta(U,V)=\xi_U=d$. Similarly, $\xi_V=d$.
\end{proof}
\begin{lemma}\label{lemma34}
 Let $S$ be  positive invertible. Then there exists a matrix $A$ such that $I-(A+A^*)=dS^2$ and \begin{equation}
  U_a=(a-1)\left(dS^2-(a-1)A^*\right)^{-1}+I=(a-1)\left(I-A-aA^*\right)^{-1}+I
  \end{equation}
  for all $a\in\mathbb T$.
\end{lemma}
\begin{proof}
  By assumption $S$ is positive invertible.  Since $S(U_a-1)S_a=(a-1)I$,
 we have $S_a=(a-1)(U_a-I)^{-1}S^{-1}$
   and hence, $S_a^*=(\bar{a}-1)S^{-1}(U_a^*-I)^{-1}. $  For any $a,b\in\mathbb T$ such that $a\ne b$. We know $S_b^*(U_b^*U_a-I)S_a=d(\bar{b}a-1)I$ by Lemma \ref{lemma33}. It follows that
 $$S_b^*((U_b^*-I)U_a+U_a-I)S_a=d((\bar{b}-1)a+a-1)I.$$ By Lemma \ref{lemma33} again,
  \begin{align}
  &S_b^*(U_b^*U_a-I)S_a\\
  \nonumber &= (\bar{b}-1)S^{-1}(U_b^*-I)^{-1}\left(U_b^* U_a -I\right)(a-1)(U_a-I)^{-1}S^{-1}\\
  \nonumber &=d(\bar{b}a-1)I.
 \end{align}
 Then
 \[(U_b^*-I)^{-1}\left(U_b^* U_a -I\right)(U_a-I)^{-1}=d(\bar{b}a-1)(\bar{b}-1)^{-1} (a-1)^{-1}S^2.\]
 Thus for any $a,b\in\mathbb T\setminus\{1\}$,
 \begin{align}
 & \ \ \ \ U_b^*-I)^{-1}\left(U_b^* U_a -I\right)(U_a-I)^{-1}\\
 \nonumber &= U_b^*-I)^{-1}\left((U_b^* -I)U_a+U_a -I\right)(U_a-I)^{-1}\\
 \nonumber &= U_a(U_a-I)^{-1}+(U_b^*-I)^{-1}\\
 \nonumber &=d\left( a(a-1)^{-1}+(\bar{b}-1)^{-1}\right)S^2.
  \end{align} This means that
  \begin{equation*}
  U_a(U_a-I)^{-1}- a(a-1)^{-1}dS^2=(\bar{b}-1)^{-1}dS^2-(U_b^*-I)^{-1}.
  \end{equation*}
  Therefore
  \begin{equation}\label{formula319}
  A=U_a(U_a-I)^{-1}- a(a-1)^{-1}dS^2=(\bar{b}-1)^{-1}dS^2-(U_b^*-I)^{-1}\end{equation} is a matrix independent of $a$ and $b$ with
      $A^*=(a-1)^{-1}dS^2-(U_a-I)^{-1}$.
  Thus
  \[A+A^*=U_a(U_a-I)^{-1}- a(a-1)^{-1}dS^2+(a-1)^{-1}dS^2-(U_a-I)^{-1}=I-dS^2.\] Consequently, $  (I-(A+A^*)=dS^2$ and
  \[
  U_a=(a-1)\left(dS^2-(a-1)A^*\right)^{-1}+I.\]
  \end{proof}
  Upon replacing $S$ and $S_U$ by $\sqrt{|d|}\,S$ and $\frac{1}{\sqrt{|d|}}S_U$, respectively, we now summarise the necessary conditions for a bijection defined on $\operatorname{PI}(M_n(\mathbb{C}))$ to be an order automorphism.

\begin{proposition}\label{prop32}
Let $\Phi$ be an order automorphism of $\operatorname{PI}(M_n(\mathbb{C}))$ satisfying $\Phi(I) = I$. Then there exist a unitary or conjugate unitary operator $U$ acting on $\mathbb{C}^n$, together with a matrix $A\in M_n(\mathbb{C})$ such that $I-(A+A^*)$ is either positive invertible or negative invertible, with the property
\[
\Phi = \Delta_{UU^*} \circ \Psi_A,
\]
where $\Psi_A$ denotes the associated order automorphism characterised by the following four conditions:

   $(1)$ $I-(A+A^*) = \varepsilon S^2$, where $\varepsilon\in\{1,-1\}$ and $S$ is a positive invertible matrix.

    $(2)$
     For every $a\in\mathbb{T}$,
    \[
    \Psi_A(aI) = U_a = (a-1)\bigl(I - A - aA^*\bigr)^{-1} + I.
    \]

    $(3)$ There exists a family of invertible matrices $\{S_a: a\in\mathbb{T}\}$ obeying $S_1 = S$, such that the action of $\Psi_A$ is given explicitly as follows.
    For any  unitary matrix $U = \sum_{k=1}^n a_k e_k\otimes e_k$ and any vector $x = \sum_{k=1}^n x_k e_k\in\mathbb{C}^n$, we have
        \begin{equation}\label{prop34formula1}
        S_U x = \sum_{a_k\neq 1} x_k S_{a_k} e_k + \varepsilon\,S\sum_{a_k=1} x_k e_k, \end{equation} and
    \begin{equation} \label{prop34formula2}   \Phi(U)S_U x = \sum_{a_k\neq 1} x_k U_{a_k} S_{a_k} e_k + \varepsilon\,S\sum_{a_k=1} x_k e_k.
    \end{equation}

    $(4)$  The identities
    \[
    S_U = S - A^*S^{-1}(U-I),\qquad
    S\bigl(\Phi(U)-I\bigr)S_U = U-I
    \]
    and
    \[
    S_U^*\bigl(\Phi(U)^*\Phi(V)-I\bigr) S_V = \varepsilon\bigl(U^*V-I\bigr)
    \]
    hold for all non-identity unitary matrices $U,V$.
\end{proposition}
\begin{proof}
By formula (\ref{formula31}) and Lemma \ref{lemma31}, we have $\Phi(aE) = U_a P_{S_a E}$ for some invertible linear or conjugate-linear transformation $S_a$ and every projection $E\in\mathcal P(M_n(\mathbb{C}))$.
Let $S_1 = U|S_1|$ denote the polar decomposition of $S_1$. In view of Remark \ref{remark31}, the map $\Psi = \Delta_{U^*U} \circ \Phi$ is an order automorphism satisfying
\[
\Psi(E) = U^* P_{SE} U = P_{U^* SE} = P_{|S|E}
\]
for all projections $E$. Consequently, $\Phi = \Delta_{UU^*} \circ \Psi$. We may therefore assume without loss of generality that $S$ is positive invertible.

Substitute $S \mapsto \sqrt{|d|}\,S$ and $S_U \mapsto \frac{1}{\sqrt{|d|}}S_U$ into the statements of Lemma \ref{lemma34}; then assertions (1) and (2) are immediately valid. Analogously, items (3) and (4) follow from Lemma \ref{lemma33} under the same substitution rule for $S$ and $S_U$.
\end{proof}

We next prove that each matrix $A$ satisfying condition (1)   of Proposition \ref{prop32} uniquely determines an automorphism $\Psi_A$ of $\operatorname{PI}(M_n(\mathbb{C}))$ for which $\Psi_A$ satisfies conditions (2), (3) and (4) of Proposition \ref{prop32}.
\begin{lemma}\label{lemma35}
  Suppose $A \in M_n(\mathbb C)$ satisfies
\(
I - (A + A^*)
\)
  is an   invertible matrix. Then for any $a \in \mathbb{T}$, the matrix
\[
I - A  - aA^*
\]
is invertible, and the matrix
\[
U_a = (a - 1)\bigl(I - A  - aA^*\bigr)^{-1} + I=(a-1)\bigl(I-(A+A^*)-(A-1)A^*\bigr)^{-1}+I
\]
is unitary.
\end{lemma}
\begin{proof}
If $I-A-aA^*$ is not invertible for some $a\in\mathbb T$, then $(I-A-aA^*)x=0$ for a unit vector $x$. That is, $x=Ax+aA^*x$. Hence
$\langle Ax,x\rangle +a\langle A^*x,x\rangle=\langle x,x\rangle =1$. Put $t=\langle Ax,x\rangle. $
We have $t+a\bar{t}=\bar{t}+\bar{a}t=1$. Thus $a\bar{t}+t=a=1$, a contradiction. Thus $I-A-aA^*=I-(A+A^*)-A^*(a-1)$ is invertible for all $a\in\mathbb T$.

Put  $C= a(I-A^*)-A$ and $D= I-A-aA^* $.   Then $U_a=CD^{-1}$.
   Now
   \begin{align*}
    C^*C&= (a(I-A^*)-A)^*(a(I-A^*)-A)\\
   &=(I-A)(I-A^*)-a(A^*(I-A^*))-\bar{a}(I-A)A+A^*A\\
   &=I-A-A^*+AA^*-aA^*(I-A^*)-\bar{a}(I-A)A+A^*A.
   \end{align*}
   \begin{align*}
    D^*D&=(-aA^*+(I-A))^*(-aA^*+(I-A))  \\
   &= (-\bar{a}A+(I-A^*))(-aA^*+(I-A))\\
   &= AA^*-\bar{a}A(I-A)-a(I-A^*)A^*+I-A^*-A+A^*A.
   \end{align*}
   Thus $C^*C=D^*D$, which implies that $U_a$ is unitary.
\end{proof}
 \begin{lemma}   Assume that
\(
I - (A + A^*) = \varepsilon S^2,
\)
where $\varepsilon \in \{1, -1\}$ and $S$ is a positive invertible matrix. Then for every unitary matrix $W$, there exists a matrix $S_W$, together with a map $\Psi_A$ defined on the set of all unitary elements contained in $\operatorname{PI}(M_n(\mathbb{C}))$, such that the pair $(S_W, \Psi_A)$ fulfills the assertions of Proposition \ref{prop32}, with $U_a$ for any $a\in\mathbb T$ given as in Lemma \ref{lemma35}.
          \end{lemma}
     \begin{proof}
  Let $a\in\mathbb T\setminus\{1\}$. It follows from Lemma \ref{lemma35}  that $U_a-I$ is invertible  with inverse
\begin{equation}\label{formula322}
(U_a-I)^{-1}=\frac{1}{a-1}\varepsilon S^2-A^*=\frac{1}{a-1}\varepsilon S^2-(I-A-\varepsilon S^2)=\frac{a}{a-1}\varepsilon S^2-I+A.\end{equation}
Thus
\begin{equation}\label{formula323}
(U_a-I)^{-1}=\frac{1}{a-1}\varepsilon S^2-A^* =\frac{1}{a-1}\left(\varepsilon S-(a-1)A^*S^{-1}\right)S .\end{equation}
Put
\begin{equation}\label{formula324}
S_a=\varepsilon S-(a-1)A^*S^{-1}.\end{equation}
 Then $S_a$ is invertible and
\begin{equation}\label{formula325}
S(U_a-I)S_a=(a-1)I.
\end{equation}In particular,
\begin{equation}\label{formula326}
S_a=(a-1)(U_a-I)^{-1}S^{-1}.
\end{equation}
Moreover,
\begin{equation}\label{formula327}
U_a(U_a-I)^{-1}=(U_a-I)^{-1}+I=\frac{a}{a-1}\varepsilon S^2+A.\end{equation}
 In particular,
 \begin{equation}\label{formula328}
 A=U_a(U_a-I)^{-1}-\frac{a}{a-1}\varepsilon S^2=\frac1{\bar{b}-1}\varepsilon S^2-(U_b^*-I)^{-1},\ \ \  \forall a,b\in\mathbb T\setminus\{1\}.
 \end{equation}
It follows that
\begin{equation}\label{formula329}
U_a(U_a-I)^{-1}+(U_b^*-I)^{-1}=\left(\frac{a}{a-1}+\frac1{\bar{b}-1}\right)\varepsilon S^2, \ \ \  \forall a,b\in\mathbb T\setminus\{1\}.
\end{equation}
Thus
\begin{align}\label{formula330}
&\ \ \ \  (U_b^*-1)^{-1}\left(U_b^*U_a-I\right)(U_a-I)^{-1}\\
\nonumber & =(U_b^*-1)^{-1}\left((U_b^*-I)U_a+(U_a-I)\right)(U_a-I)^{-1}\\
\nonumber  &=U_a(U_a-I)^{-1}+(U_b^*-I)^{-1}\\
\nonumber  &=\left(\frac{a}{a-1}+\frac1{\bar{b}-1}\right)\varepsilon S^2\\
\nonumber &=\frac{(\bar{b}-1)a+(a-1)}{(\bar{b}-1)(a-1)}\varepsilon S^2\\
\nonumber &=\frac{\bar{b}a-1}{(\bar{b}-1)(a-1)}\varepsilon S^2.
\end{align}
This implies that
\begin{align}\label{formula331}
& S_b^*(U_b^*U_a-I)S_a\\
\nonumber &=\left((\bar{b}-1)S^{-1}(U_b^*-I)^{-1}\right)(U_b^*U_a-I)\left((a-1)(U_a-I)^{-1}(S)^{-1}\right)\\
\nonumber &=\varepsilon (\bar{b}a-1)I.
\end{align}
  Let $W$ be a non-identity  unitary matrix. Then $W=\sum\limits_{k=1}^n a_ke_k\otimes e_k$ for some $a_k\in\mathbb T$ and an orthogonal basis $\{e_k:1\leq k\leq n\}$.
  For any $x=\sum\limits_{k=1}^n x_ke_k\in\mathbb C^n$, we define
  \begin{equation}\label{formula332}
  S_Wx=\sum\limits_{a_k\ne 1} x_kS_{a_k}e_k+\varepsilon S\sum\limits_{a_k= 1} x_ke_k,
   \end{equation}
   \begin{equation}\label{formula333}
  \Psi_A(W)S_Wx=\sum\limits_{a_k\ne 1}x_kU_{a_k}S_{a_k}e_k+\varepsilon S\sum\limits_{a_k= 1} x_ke_k
  \end{equation}
  and   $\Psi_A(W)z=0$ for any $z\in R(S_W)^{\bot}$.

 By formula (\ref{formula324}), for all $x\in\mathbb C^n$,
\begin{align*}
S_Wx&=\sum\limits_{a_k\ne 1}x_kS_{a_k}e_k+\sum\limits_{a_k= 1}\varepsilon x_kSe_k\\
&=\sum\limits_{a_k\ne 1}x_k\left( \varepsilon S-A^*S^{-1}(a_k-1)\right)  e_k +\sum\limits_{a_k=1}x_k   \varepsilon Se_k\\
&=(\varepsilon S-A^*S^{-1}(W-I))x.
\end{align*}
Thus, $S_W=\varepsilon S-A^*S^{-1}(W-I)$.

   It is elementary that $\Psi_A(aI)=U_a$ for all $a\in\mathbb T$.  Take any two non-identity  unitary matrices $W$ and $V$ in $M_n(\mathbb C)$. Let $V=\sum\limits_{j=1}^nb_j f_j\otimes f_j$ for some $b_j\in \mathbb T(1\leq j\leq n)$ and  an orthogonal  basis $\{f_j:1\leq j\leq n\}$.
  For any   $x,y\in\mathbb C^n$ with $x=\sum\limits_{k=1}^nx_ke_k$ and $y=\sum\limits_{j=1}^ny_jf_j$,
  \begin{align*}\label{formula335}
  & \ \    \ \ \langle S_V^*(\Psi_A(V))^*\Psi_A(W)S_Wx, y\rangle \\
  \nonumber & =\langle \Psi_A(W)S_Wx, \Psi_A(V)S_Vy\rangle\\
  \nonumber &=\langle \sum\limits_{a_k\ne 1}x_kU_{a_k}S_{a_k}e_k+\varepsilon S\sum\limits_{a_k= 1} x_ke_k, \sum\limits_{b_j\ne 1}y_jU_{b_j}S_{b_j}f_j+\varepsilon S\sum\limits_{b_j= 1} y_jf_j\rangle\\
  \nonumber &=\sum\limits_{a_k\ne1,b_j\ne1} x_k\overline{y_j}\langle U_{a_k}S_{a_k}e_k,U_{b_j}S_{b_j}f_j\rangle
  +\sum\limits_{a_k=1,b_j\ne1} x_k\overline{y_j}\langle \varepsilon S e_k,U_{b_j}S_{b_j}f_j\rangle\\
  \nonumber &\ \ \ \ \ \sum\limits_{a_k\ne1,b_j=1} x_k\overline{y_j}\langle U_{a_k}S_{a_k}e_k,\varepsilon Sf_j\rangle+
  \sum\limits_{a_k=1,b_j=1} x_k\overline{y_j}\langle \varepsilon  S e_k,\varepsilon  S f_j\rangle\\
\nonumber
 &=\sum\limits_{a_k\ne1,b_j\ne1}x_k\overline{y_j}\langle S_{b_j}^*U_{b_j}^* U_{a_k}S_{a_k}e_k, f_j\rangle+\sum\limits_{a_k=1,b_j\ne1}x_k\overline{y_j}\langle \varepsilon S_{b_j}^*U_{b_j}^*  S e_k, f_j\rangle\\
\nonumber &\ \ \ \ \ +\sum\limits_{a_k\ne1,b_j=1}x_k\overline{y_j}\langle \varepsilon S U_{a_k} S_{a_k} e_k, f_j\rangle
+\sum\limits_{a_k=1,b_j=1}x_k\overline{y_j}\langle  S    e_k, Sf_j\rangle\\
\nonumber
&=\sum\limits_{a_k\ne1,b_j\ne1}x_k\overline{y_j}\langle\left( S_{b_j}^* S_{a_k}+\varepsilon (\overline{b_j}a_k-1)I\right)e_k, f_j\rangle\\
\nonumber &\ \ \ \ \ +\sum\limits_{a_k=1,b_j\ne1}x_k\overline{y_j}\langle \varepsilon \left(S_{b_j}^*   S +(\overline{b_j}-1)I\right)e_k, f_j\rangle\\
\nonumber &\ \ \ \ \ +\sum\limits_{a_k\ne1,b_j=1}x_k\overline{y_j}\langle \varepsilon\left( S   S_{a_k}+(a_k-1)I\right) e_k, f_j\rangle
+\sum\limits_{a_k=1,b_j=1}x_k\overline{y_j}\langle  S    e_k, Sf_j\rangle\\
\nonumber &= \sum\limits_{a_k\ne1,b_j\ne1}x_k\overline{y_j}\langle S_{b_j}^* S_{a_k}e_k,f_j\rangle +\sum\limits_{a_k\ne1,b_j\ne1}x_k\overline{y_j}\langle \varepsilon (\overline{b_j}a_k-1)I)e_k, f_j\rangle\\
\nonumber &\ \ \ \ \ +\sum\limits_{a_k=1,b_j\ne1}x_k\overline{y_j}\langle \varepsilon  S_{b_j}^*   S e_k,f_j\rangle
+\sum\limits_{a_k=1,b_j\ne1}x_k\overline{y_j}\langle \varepsilon(\overline{b_j}-1)Ie_k, f_j\rangle\\
\nonumber &\ \ \ \ \ +\sum\limits_{a_k\ne1,b_j=1}x_k\overline{y_j}\langle \varepsilon S   S_{a_k}e_k,f_j\rangle
+\sum\limits_{a_k\ne1,b_j=1}x_k\overline{y_j}\langle \varepsilon(a_k-1)I e_k, f_j\rangle\\
\nonumber &\ \ \ \ \ +\sum\limits_{a_k=1,b_j=1}x_k\overline{y_j}\langle  S    e_k, Sf_j\rangle\\
\nonumber &=\left(\sum\limits_{a_k\ne1,b_j\ne1}x_k\overline{y_j}\langle S_{b_j}^* S_{a_k}e_k,f_j\rangle +\sum\limits_{a_k=1,b_j\ne1}x_k\overline{y_j}\langle \varepsilon  S_{b_j}^*   S e_k,f_j\rangle \right. \\
\nonumber &\ \ \ \ \  \left.  +\sum\limits_{a_k\ne1,b_j=1}x_k\overline{y_j}\langle \varepsilon S   S_{a_k}e_k,f_j\rangle +\sum\limits_{a_k=1,b_j=1}x_k\overline{y_j}\langle  S    e_k, Sf_j\rangle\right)\\
\nonumber & +\left( \sum\limits_{a_k\ne1,b_j\ne1}x_k\overline{y_j}\langle \varepsilon (\overline{b_j}a_k-1)I)e_k, f_j\rangle+
\sum\limits_{a_k=1,b_j\ne1}x_k\overline{y_j}\langle \varepsilon(\overline{b_j}-1)Ie_k, f_j\rangle\right.\\
\nonumber &\ \ \ \ \ \left. +\sum\limits_{a_k\ne1,b_j=1}x_k\overline{y_j}\langle \varepsilon(a_k-1)I e_k, f_j\rangle\right)\\
\nonumber &=\langle\left( S_V^* S_W+\varepsilon (W^*V-1)I\right)x, y\rangle.
  \end{align*}
Thus
\begin{equation*}
S_V^*\left((\Psi_A(V))^*\Psi_A(W)-I\right)S_W=\varepsilon (V^*W-I),\end{equation*} that is, Assertion (4) in Proposition \ref{prop32} holds.
Take any $V$ such that $ V^*W-I$ is invertible. This  implies that $S_W$ is invertible by  above formula. It also follows that $(\Psi_A(W))^*\Psi_A(W)=I$ and $\Psi_A(W)$ is unitary.
\end{proof}
 We now extend $\Psi_A$ to $\operatorname{PI}(M_n(\mathbb C))$.  Note that Assertion (4)  of Proposition \ref{prop32} holds.   For any unitary matrices $W,V\in \operatorname{PI}(M_n(\mathbb C))$, $S_WE=S_VE$ if $WE=VE$.
Now for any partial isometry $R$ with initial projection $E_R=R^*R$, we have  $R\overset{*}{\leq}W$, that is,  $R=WE_R$ for some unitary matrix $W\in M_n(\mathbb C)$. If $R=WE_R=VE_R$ for two unitary matrices $W$ and $V$, then  $S_WE_R=S_VE_R$ and
\begin{equation}\label{formula336}
S_V^*\left((\Psi_A(V))^*\Psi_A(W)-I\right)S_WE_R=0 \end{equation}
 by  Assertion (4)  of Proposition \ref{prop32} again.  Consequently,
\[ (\Psi_A(V))-\Psi_A(W))P_{S_WE_R}=0. \] This means that $\Psi_A(W)P_{S_WE_R}=\Psi_A(V)P_{S_VE_R}$ is a well-defined partial isometry.
 We define
 \begin{equation}\label{formula337}
 \Psi_A(R)= \Psi_A(W)P_{S_WE_R}.
 \end{equation}
  \begin{lemma}\label{lemma37} The map $\Psi_A$ defined via formula (\ref{formula337}) is an order automorphism of $\operatorname{PI}(M_n(\mathbb{C}))$.
  \end{lemma}
  \begin{proof}
   $\Psi_A$ is a map on $\operatorname{PI}(M_n(\mathbb C))$. In particular, for any rank 1 partial isometry $y\otimes x$, $\Psi_A(y\otimes x)=\Psi_A(W)\frac{S_Wx}{\|S_Wx\|}\otimes\frac{ S_Wx}{\|S_Wx\|}$ for all unitary matrix $W$ with $Wx=y$.
 Thus
 \begin{equation}\label{formula338}
 \Psi_A(R)= \Psi_A(W)P_{S_WE_R}=\sup\left\{\Psi_A(W)\frac{S_Wx}{\|S_Wx\|}\otimes\frac{ S_Wx}{\|S_Wx\|}:x\in E_R(\mathbb C^n)\right\}.
   \end{equation}
   For any two partial isometries $R$ and $T$ with $R\leq T$, we have $R\leq  T\leq W$ for a unitary matrix $W$.
   Thus $\Psi_A(R)\leq \Psi_A(T)$ by formula (\ref{formula338}).

 We next prove that $\Psi_A$ is  surjective. For any unitary matrix $U\in M_n(\mathbb C)$, put $C=I+\varepsilon S(U-I)S+S(U-I)A^*S^{-1}$ and $D=I+S(U-I)A^*S^{-1}$.   We claim that  $CC^*=DD^*$.

  Note that $C=\varepsilon S(U-I)S+D$. It is sufficient to prove that
$$S(U-I)S^2(U^*-I)S+ \varepsilon S(U-I)SD^*+\varepsilon DS(U^*-I)S=0.$$
Now \begin{align*}
&\ \ \ \ S(U-I)S^2(U^*-I)S+ \varepsilon S(U-I)SD^*+\varepsilon DS(U^*-I)S\\
&=S(U-I)S^2(U^*-I)S+\varepsilon  S(U-I)S+\varepsilon S(U-I)A(U^*-I)S\\
&\ \ \ \ \   +\varepsilon S(U^*-I)S+\varepsilon S(U-I)A^*(U^*-I)S\\
&=\varepsilon S(U-I)(I-(A+A^*))(U^*-I)S+\varepsilon S(U-I)S+\varepsilon S(U-I)A(U^*-I)S\\
&\ \ \ \ \   +\varepsilon S(U^*-I)S+ \varepsilon S(U-I)A^*(U^*-I)S\\
&=\varepsilon \big(S(U-I)(U^*-I)S+S(U-I)S+S(U^*-I)S\big)\\
&=\varepsilon \big( S(2-U-U^*)S+S(U+U^*-2)S\big)\\
&=0.
\end{align*}
Thus there exists a unitary matrix $W$ such that $C=DW$. That is,
$$S(U-I)S=(I+S(U-I)A^*S^{-1})(W-I).   $$ Consequently,
$$S(U-I)S_W=S(U-I)(S-A^*S^{-1}(W-I))=W-I.$$
By  Assertion (4) of Proposition \ref{prop32}, $S(\Psi_A(W)-I)S_W=W-I=S(U-I)S_W$.
It follows that $\Psi_A(W)=U$. Hence $\Psi_A$ is surjective on the set of unitary matrices.
If $V\leq U$ is a partial isometry, then $V=UV^*V=UE_V$. Let $F$ be the projection  on $S_W^{-1}(E_V)$. Then $S_WF=E_V=V^*V$ and $\Psi_{A}(WF)=\Psi_A(W)P_{S_WF}=UE_V=V$. Thus $\Psi_A$ is an order isomorphism.
\end{proof}
 We now summarize our results into the following theorem.
\begin{theorem}
Let $\varphi$ be  a bijection  on $\operatorname{PI}(M_n(\mathbb C))$. Then $\varphi$ is an order isomorphism if and only if there exist two unitary or  two conjugate unitary transformations $C$ and $D$ on $\mathbb C^n$ and a matrix $A\in M_n(\mathbb C)$ satisfying   $I- (A+A^*)$ is  positive or negative  invertible such that
$\varphi=\Delta_{CD}\circ \Psi_A$.
\end{theorem}
\begin{proof} The sufficiency follows from  Proposition \ref{prop32}.

Let $\varphi$ be an order automorphism of $\operatorname{PI}(M_n(\mathbb C))$. Put $U=\varphi(I)$ and $\Phi=\Delta_{U^*I}\circ \varphi$. Then $\Phi$ is an order automorphism with $\Phi(I)=I$.   By Lemma \ref{lemma31}, we obtain a family of invertible transformations $\{\tilde{S}_a\colon a\in\mathbb T\}$ on $\mathbb C^n$ such that either every $\tilde{S}_a$ ($a\in\mathbb T$) is linear, or every $\tilde{S}_a$ is conjugate linear.  Let $\tilde{S}_1=W_1|\tilde{S}_1|$ be the decomposition of $\tilde{S}_1$. Then $W$ is linear or conjugate and $|\tilde{S}_1|$ is positive invertible.  Put $\Psi=\Delta_{W_1^*W_1}\circ \Phi$. Then $\Psi(E)=\Delta_{W^*W}\circ \Phi(E)=W_1^*P_{\tilde{S}_1E}W_1=P_{|\tilde{S}_1|E}$ for every projection $E\in \mathcal P(M_n(\mathbb C))$.  Thus, the family of invertible transformations $\{S_a\colon a\in\mathbb T\}$ associated with $\Psi$ consists of linear maps, and $S=|\tilde{S}_1|$ is a positive invertible matrix.
By Lemma \ref{lemma37},  we have $\Psi=\Psi_A$, where $A$ satisfies $(1)$ and $ (2)$ of Proposition
\ref{prop32}. Put $C=UW_1$ and $D=W_1^*$. We have $\varphi=\Delta_{CD}\circ \Psi_A$.
\end{proof}
   \begin{example} Take $A=0$ Then $\Psi_0$ is the identity. If $A=I$, then $\Psi_I(V)=V^*$, $\forall V\in \operatorname{PI}(M_n(\mathbb C))$.
   \end{example}
Finally, we conclude that an automorphism $\varphi$ of $\operatorname{PI}(M_n(\mathbb C))$ preserves orthogonality if and only if the corresponding matrix $A$ is either $0$ or $I$. Consequently, up to unitary or conjugate unitary transformations, there exist exactly two order-orthogonal  automorphisms.

\end{document}